\documentclass[amscd,amssymb,verbatim,11pt]{amsart}
\usepackage{epsfig}

\newtheorem{theorem}{Theorem}[section]

\begin{document}

\title{The finite Fourier transform of classical polynomials}


\author[A. Dixit]{Atul Dixit}
\address{Department of Mathematics,
Tulane University, New Orleans, LA 70118}
\email{adixit@tulane.edu}

\author[L. Jiu]{Lin Jiu}
\address{Department of Mathematics,
Tulane University, New Orleans, LA 70118}
\email{ljiu@tulane.edu}

\author[V. H. Moll]{Victor H. Moll}
\address{Department of Mathematics,
Tulane University, New Orleans, LA 70118}
\email{vhm@tulane.edu}

\author[C. Vignat]{Christophe Vignat}
\address{Department of Mathematics,
Tulane University, New Orleans, LA 70118}
\email{cvignat@math.tulane.edu}

\subjclass{Primary 42A16, Secondary 33C45}

\date{\today}

\keywords{Fourier transform, orthogonal polynomials, Jacobi polynomials, }

\begin{abstract}
The finite Fourier transform of a family of  orthogonal polynomials $A_{n}(x)$, is the 
usual transform of the polynomial extended 
by $0$ outside their natural domain. Explicit expressions are given for the Legendre, 
Jacobi, Gegenbauer and Chebyshev families. 
\end{abstract}

\maketitle

\newcommand{\nn}{\nonumber}
\newcommand{\ba}{\begin{eqnarray}}
\newcommand{\ea}{\end{eqnarray}}
\newcommand{\no}{\noindent}
\newcommand{\la}{\lambda}
\newcommand{\ch}{}
\newcommand{\realpart}{\mathop{\rm Re}\nolimits}
\newcommand{\imagpart}{\mathop{\rm Im}\nolimits}
\newcommand{\pFq}[5]{\ensuremath{{}_{#1}F_{#2} \left( \genfrac{}{}{0pt}{}{#3}
{#4} \bigg| {#5} \right)}}

\newtheorem{Definition}{\bf Definition}[section]
\newtheorem{Thm}[Definition]{\bf Theorem}
\newtheorem{Example}[Definition]{\bf Example}
\newtheorem{Lem}[Definition]{\bf Lemma}
\newtheorem{Note}[Definition]{\bf Note}
\newtheorem{Cor}[Definition]{\bf Corollary}
\newtheorem{Conj}[Definition]{\bf Conjecture}
\newtheorem{Prop}[Definition]{\bf Proposition}
\newtheorem{Problem}[Definition]{\bf Problem}
\numberwithin{equation}{section}

\section{Introduction} \label{sec-intro}
\setcounter{equation}{0}

Compendia of formulas, such as the classical \textit{Table of {I}ntegrals, {S}eries and 
{P}roducts} by I.~S.~Gradshteyn and I.~M.~Ryzhik \cite{gradshteyn-2007a} 
and the recent \textit{NIST {H}andbook of 
{M}athematical {F}unctions} \cite{olver-2010a} do not contain a systematic collection of 
Fourier transforms of orthogonal polynomials. 

Special cases do appear. For instance, \cite[formula $18.17.19$]{olver-2010a} contains the identity
\begin{equation}
\int_{-1}^{1} P_{n}(x) e^{\imath \la x} dx = \imath^{n} \sqrt{ \frac{2 \pi}{\la}} J_{n+\tfrac{1}{2}}(\la),
\end{equation}
\noindent
for the \textit{finite Fourier transform} of the Legendre polynomial $P_{n}$. Here $J_{\alpha}$ is 
the Bessel function defined by 
\begin{equation}
J_{\alpha}(\la) = \sum_{k=0}^{\infty} \frac{(-1)^{k} (\la/2)^{2k+ \alpha}}{k! \Gamma(k+ \alpha + 1)}.
\end{equation}
\noindent
A second example is \cite[formula $3.3(7)$, page 123]{erderly-1954a}
\begin{equation}
\int_{-1}^{1} P_{\nu}(x) e^{\imath \la x} dx = \frac{2 \pi \sin \pi \nu}{\nu(\nu+1)} e^{- \imath \la} 
\pFq22{1,1}{- \nu, 2 + \nu}{2 \imath \la}.
\end{equation}

The more natural situation, where the corresponding kernel appears in the integrand, is included 
in the tables. For instance, for the Jacobi polynomial, \cite[18.17.16]{olver-2010a} gives 
\begin{equation}
\label{jacobi-101}
\int_{-1}^{1} (1-x)^{\alpha}(1+x)^{\beta}  P_{n}^{(\alpha, \beta)}(x) e^{ \imath \la x} \, dx   = 
X_{n}(\la;\alpha,\beta) \pFq11{n+\alpha + 1}{2n+\alpha + \beta + 2}{-2 \imath \la}, 
\end{equation}
\noindent
with 
\begin{equation}
X_{n}(\la;\alpha,\beta) = \frac{(\imath \la)^{n} e^{\imath \la}}{n!} 2^{n+ \alpha + \beta +1}  \times B(n + \alpha + 1, n + \beta + 1) 
\end{equation}

The work presented here was stimulated by results of A.~Fokas et al. \cite{fokas-2014a}. A second 
motivation was the fact that the authors were unable to find the 
finite Fourier transform of classical orthogonal polynomials readily available in the 
literature. These results were also developed in \cite{fokas-2014a} 
and some of them  appear in \cite{greenen-2008a}. The authors wish to thank 
A. Fokas and T. Koorwinder for correspondence on the questions discussed here.

The goal of this project  is to produce  closed-form evaluations of definite 
integrals of the form 
\begin{equation}
\widehat{P}(\la):= \int_{a}^{b} e^{\imath \la x} P(x) \, dx 
\end{equation}
\noindent 
for a variety of polynomials $P$, orthogonal on the interval $[a,b]$. The function $\widehat{P}(\la)$ is 
called the \textit{finite Fourier transform} of the polynomial $P$. The case considered here includes 
the Legendre polynomial $P_{n}(x)$,  the Jacobi polynomial $P_{n}^{(\alpha, \beta)}(x)$, from which 
the Gegenbauer polynomials $C_{n}^{(\nu)}(x)$ and both types of Chebyshev polynomials 
$T_{n}(x)$ and $U_{n}(x)$ are derived. 

Naturally, depending on the representation given of the polynomial $P$, it is possible 
to obtain a variety of expressions for $\widehat{P}$.  For instance, if an expression 
for the coefficients of $P$ is available, the identity in Lemma \ref{lemma-intpower} 
and a simple scaling give directly a double-sum representation for $\widehat{P}(\la)$. 

\smallskip

It is convenient to introduce the notation 
\begin{equation}
E_{n}(x) = \sum_{j=0}^{n} \frac{x^{j}}{j!}
\end{equation}
\noindent 
for the partial sums of the exponential function. Many of the results may be 
expressed in terms of $E_{n}$. The result is elementary and it appears in 
\cite[formula $2.323$]{gradshteyn-2007a}.

\begin{Lem}
\label{lemma-intpower}
Let $k \geq 0$ be an integer and $\la$ an indeterminate. Then,
\begin{equation}
\int_{-1}^{1} x^{k} e^{\imath \la  x} dx =\frac{(-1)^{k}  k!}{(\imath \la)^{k+1} }
\left[ e^{\imath \la} E_{k}(- \imath \la) -e^{- \imath \la} E_{k}(\imath \la) \right], \ch
\end{equation}
\noindent 
and 
\begin{equation}
\int_{0}^{1} x^{k} e^{\imath \la  x} dx =\frac{(-1)^{k}  k!}{(\imath \la)^{k+1} }
\left[ e^{\imath \la} E_{k}(- \imath \la) - 1  \right]. \ch
\end{equation}
\end{Lem}
\begin{proof}
Integrate by parts. 
\end{proof}

\begin{Note}
The notation is standard. The symbol $(a)_{n}$ denotes the shifted factorial, defined by  
$(a)_{n} = a(a+1) \cdots (a+n-1)$ and $(a)_{0} = 1$. The elementary 
properties 
\begin{eqnarray}
(1)_{n} & = & n!  \label{poch-110} \\
(a)_{n} & = &  \frac{\Gamma(a+n)}{\Gamma(a)} \label{poch-111} \\
( a + \tfrac{1}{2} )_{n} & = &  \frac{(2a)_{2n}}{2^{2n} (a)_{n}}    \label{poch-3}  \\
(-n)_{k} & = & \frac{(-1)^{k}n!}{(n-k)!} \qquad \text{ for }n, \, k \in \mathbb{N}    \label{poch-112}   \\
(n+1)_{k} & = & \frac{(n+k)!}{n!} \qquad \text{ for }n, \, k \in \mathbb{N},      \label{poch-113} \\
(-a)_{n} & = & (-1)^{n} (a-n+1)_{n}, \label{poch-114}
\end{eqnarray}
\noindent
are used throughout. 
\end{Note}

\section{Legendre polynomials}
\label{sec-legendre}

This section contains a variety of formulas for the finite Fourier transform of the Legendre polynomials
$P_{n}(x)$. These are orthogonal polynomials on the interval $[-1,1]$, with weight $w(x) \equiv 1$. 
The next theorem gives all the results. 

\begin{Thm}
\label{fourier-legendre}
The finite Fourier transform of the Legendre polynomial $P_{n}(x)$ is given by one of the 
four equivalent forms: 

\smallskip

\begin{eqnarray*}
\widehat{P_{n}}(\la) & = & 
2^{n} \sum_{k=0}^{n} \binom{n}{k} \binom{\tfrac{1}{2}(n+k-1) }{n} 
\frac{(-1)^{k} k! }{(\imath \la)^{k+1}} 
\left[ e^{\imath \la} E_{k}(- \imath \la) - e^{- \imath \la } E_{k}(\imath \la) \right]  \ch \\
& = &  \imath^{n} \sqrt{\frac{2 \pi}{\la}} J_{n+1/2}(\la)  \ch \\
& = & 2 \sum_{k=0}^{n} \frac{(n+k)!}{(n-k)! \, k! } 
\frac{\left[e^{-\imath \la}E_{k}(2 \imath \la) - e^{\imath \la} \right]}{(-2 \imath \lambda)^{k+1} }  \ch \\
& = & 2 \sum_{k=0}^{n} \frac{(n+k)!}{(n-k)! \, k!} 
\frac{\left[ (-1)^{n+k} e^{ -\imath \lambda} - e^{ \imath \lambda} \right]}{(-2 \imath \la)^{k+1}}.  \ch
\end{eqnarray*}

\end{Thm}
\begin{proof}
The first formula follows from the explicit representation 
\begin{equation}
P_{n}(x) = 2^{n} \sum_{k=0}^{n} \binom{n}{k} \binom{ \tfrac{1}{2}(n+k-1)}{n} x^{k} \ch
\end{equation}
\noindent
given in ?? and Lemma \ref{lemma-intpower}. The second expression for $\widehat{P}_{n}(\la)$ 
comes from their Rodrigues formula
\begin{equation}
P_{n}(x) = \frac{1}{2^{n} \, n!} \left( \frac{d}{dx} \right)^{n} (x^{2}-1)^{n},  \ch
\end{equation}
(see \cite[Formula $8.910.2$]{gradshteyn-2007a}) and it appears as  entry $7.242.5$ in \cite{gradshteyn-2007a}. Then
\begin{equation}
\widehat{P}_{n}(\lambda)  = \frac{1}{2^{n}n!} \int_{-1}^{1} e^{\imath  \la x} \left( \frac{d}{dx} \right)^{n} (x^{2}-1)^{n} \, dx
\end{equation}
\noindent
and integrating  by parts $n$-times yields 
\begin{equation}
\widehat{P}_{n}(\lambda)  = \frac{ (-\imath \lambda)^{n}}{2^{n} n!} \int_{-1}^{1} (x^{2}-1)^{n} e^{\imath \la x  }  \, dx.   \ch
\end{equation}
\noindent
Entry $3.387.2$ of \cite{gradshteyn-2007a} states that 
\begin{equation}
\int_{-1}^{1} (1-x^{2})^{\nu-1} e^{ \imath \mu x} \, dx = \sqrt{\pi} 
\left( \frac{2}{\mu} \right)^{\nu - \tfrac{1}{2}} \Gamma(\nu) J_{\nu - \tfrac{1}{2}}(\mu).  \ch
\end{equation}
\noindent
The result is obtained by choosing  $\mu =  \la$ and $\nu = n+1$.

\smallskip

The third form of the finite Fourier transform of the Legendre polynomials is obtained from their 
hypergeometric representation
 \begin{equation}
P_{n}(x)  =   \pFq21{-n \quad n+1}{1}{\frac{1-x}{2}}  
       =  \sum_{k=0}^{n} \frac{ (-n)_{k} (n+1)_{k}}{(1)_{k} \,  k!} \left( \frac{1-x}{2} \right)^{k}, \ch
       \label{hyper-legendre} \\
\end{equation}
\noindent
that gives

\begin{equation}
\widehat{P}_{n}(\lambda) = \sum_{k=0}^{n} \frac{(-n)_{k}(n+1)_{k}}{k!^{2}} \int_{-1}^{1} e^{ \imath \lambda x} 
\left( \frac{1 - x}{2} \right)^{k} \, dx.  \ch
\end{equation}
\noindent
A change of variables and the formulas  \eqref{poch-113} and \eqref{poch-114} give 
\begin{equation}
\label{sum-1}
\widehat{P}_{n}(\lambda) = 2 e^{\imath \lambda} \sum_{k=0}^{n} \frac{(-1)^{k} (n+k)!}{(n-k)! k!^{2}}
\int_{0}^{1} t^{k} e^{-2 \imath \lambda t} \, dt.  \ch
\end{equation}
\noindent
Lemma \ref{lemma-intpower} now gives the stated result. 

\smallskip

To produce the last form  for $\widehat{P}_{n}(\lambda)$, let 
$t = 2 \imath \lambda$ in the third expression for this transform. Then,  after 
multiplication by $t^{n}$ and some simplification, the 
claim is equivalent to the polynomial identity
\begin{equation}
\sum_{k=0}^{n} \frac{(2n-k)!}{k!(n-k)!} (-1)^{k}t^{k} \sum_{j=0}^{n-k} \frac{t^{j}}{j!} = 
\sum_{k=0}^{n} \frac{(2n-k)!}{k! (n-k)!} t^{k}. \ch
\end{equation}
\noindent
To simplify the sum, let $\nu = k+j$ on the left-hand side to show that the desired identity 
is equivalent to 
\begin{equation}
\sum_{\nu=0}^{n} \left[ \sum_{k=0}^{\nu} \frac{(-1)^{k} (2n-k)!}{k! \, (n-k)! \, (\nu-k)!} \right] t^{\nu} 
= \sum_{k=0}^{n} \frac{(2n-k)!}{k! \, (n-k)!} t^{k}.
\end{equation}
Matching coefficients, the result follows from 
\begin{equation}
\sum_{j=0}^{k} \frac{(-1)^{j} (2n-j)!}{j! \, (n-j)! \, (k-j)!} = \frac{ (2n-k)!}{k! \, (n-k)!}  \ch
\end{equation}
\noindent
for every $0 \leq k \leq n$.  This is equivalent to the binomial identity given in 
 Lemma \ref{binomial-id} below. The proof is complete.
\end{proof}

\begin{Lem}
\label{binomial-id}
For $n \in \mathbb{N}$ and $0 \leq k \leq n$
\begin{equation}
\sum_{j=0}^{k} (-1)^{j} \binom{n}{j} \binom{2n-j}{2n-k} = \binom{n}{k}. \ch
\label{binomial-1}
\end{equation}
\end{Lem}
\begin{proof}
The proof uses $\binom{r}{k}  = (-1)^{k} \binom{k-r-1}{k}$ to write 
\begin{equation}
\binom{2n-j}{2n-k} = \binom{2n-j}{k-j} = (-1)^{k-j} \binom{k-2n-1}{k-j} \ch
\end{equation}
\noindent
and then \eqref{binomial-1} is converted into Vandermonde identity 
\begin{equation}
\sum_{k=0}^{n} \binom{a}{k} \binom{b}{n-k} = \binom{a+b}{n}.  \ch
\end{equation}
\end{proof}

\section{Jacobi polynomials}
\label{sec-jacobi}

The Jacobi polynomials $P_{n}^{(\alpha, \beta)}(x)$, defined by 
\begin{equation}
P_{n}^{(\alpha, \beta)}(x) = \frac{1}{2^{n}} \sum_{k=0}^{n} \binom{\alpha + n}{k} 
\binom{\beta + n}{n-k}  (x-1)^{n-k} (x+1)^{k}   \ch
\label{jacobi-111}
\end{equation}
\noindent
are orthogonal on $[-1,1]$ with respect to the weight 
\begin{equation}
w(x) = (1-x)^{\alpha} (1+x)^{\beta}.
\end{equation}
\noindent
This section contains expressions for their finite Fourier transform.  The hypergeometric 
representation 
\begin{equation}
P_{n}^{(\alpha, \beta)}(x) = \frac{(\alpha + 1)_{n}}{n!} 
\pFq21{-n, \quad n + \alpha + \beta +1}{\alpha+1}{ \frac{1-x}{2} }, \ch
\end{equation}
\noindent
is used in the calculations.

\begin{theorem}
\label{fourier-jacobi}
The  finite Fourier transform of the Jacobi polynomials $P_{n}^{(\alpha, \beta)}(x)$ 
is given by 
\begin{eqnarray*}
\widehat{P_{n}^{(\alpha, \beta)}}(\la)  & = &    2e^{\imath \la}(\alpha+1)_{n} 
 \sum_{k=0}^{n} \frac{(n+\alpha+\beta+1)_{k}}{(n-k)! (\alpha+1)_{k}} 
\left[ \frac{e^{-2 \imath \la } E_{k}(2 \imath \la) - 1}{(-2 \imath \la)^{k+1}} \right]   \ch \\
 &=   & 2 \sum_{k=0}^{n}
\frac{(n+ \alpha + \beta + 1)_{k}}{(-2 \imath \la)^{k+1}(n-k)!} \ch \\
& & \times \left[
(-1)^{n-k}e^{- \imath \la} (\beta + k+1)_{n-k}  - e^{\imath \la} (\alpha+k+1)_{n-k} \right],
\end{eqnarray*}
\noindent
for $\la \neq 0$.  For $\la = 0$, 
\begin{equation}
\widehat{P_{n}^{(\alpha,\beta)}}(0) = \frac{(n+\alpha+\beta+1)}{2} 
\left[ \binom{\alpha + n}{n-1} -  (-1)^{n-1}\binom{\beta+n}{n-1} \right].
\end{equation}
\end{theorem}
\begin{proof}
The first statement comes from the hypergeometric form 
\begin{eqnarray}
P_{n}^{(\alpha, \beta)}(x) & = & \frac{(\alpha + 1)_{n}}{n!} 
\pFq21{-n, \quad n + \alpha + \beta +1}{\alpha+1}{ \frac{1-x}{2} }   \ch \\
 & = &  \frac{( \alpha + 1)_{n}}{n!}  
\sum_{k=0}^{n} \frac{(-n)_{k} (n + \alpha + \beta + 1)_{k}}{(\alpha + 1)_{k}  k! 2^{k}} 
( 1- x)^{k} \nonumber  \ch
\end{eqnarray}
\noindent
and  use  Lemma \ref{lemma-intpower} to produce 
\begin{equation}
\int_{-1}^{1} (1-x)^{k} e^{\imath \la x} dx = 
 -e^{\imath \la} \frac{k!}{(\imath \la)^{k+1}} \left[ e^{- 2 \imath \la} E_{k}(2 \imath \la) -1 \right]
 \end{equation}
 \noindent
 and then  $(-n)_{k} = (-1)^{k} n!/(n-k)!$ to simplify the result.

\smallskip

%
Now use  identity (the case $m=1$ of \cite[$8.961.4$]{gradshteyn-2007a}:
\begin{equation}
\label{der-jacobi}
\frac{d}{dx} P_{n}^{(\alpha, \beta)}(x) = \frac{n+\alpha + \beta +1}{2} P_{n-1}^{(\alpha + 1, \beta +1)}(x). \ch
\end{equation}
\noindent
and integrate by parts to obtain 
\begin{equation*}
\widehat{P_{n}^{(\alpha, \beta)}}(\la)  = 
 \frac{e^{\imath \la x}}{\imath \la} P_{n}^{(\alpha, \beta)}(x) 
 \Big{|}_{-1}^{1} -  
 \frac{(n+ \alpha + \beta +1)}{2 \imath  \la}
 \widehat{P_{n-1}^{(\alpha + 1, \beta + 1)}}(\la).
 \end{equation*}
 Introduce the notation for the boundary term 
 \begin{equation}
a_{n}^{(\alpha, \beta)} =  \frac{e^{\imath \la x}}{\imath \la} P_{n}^{(\alpha, \beta)}(x) \Big{|}_{-1}^{1}.
\end{equation}
\noindent
to write the previous computation as the recurrence 
\begin{equation}
\widehat{P_{n}^{(\alpha, \beta)}}(\la) = a_{n}^{(\alpha, \beta)}(\la) - 
\frac{(n + \alpha + \beta +1)}{2  \imath \la} \widehat{P_{n-1}^{(\alpha + 1, \beta +1)}}(\la).
\end{equation}
\noindent
Iteration yields 
\begin{eqnarray*}
\widehat{P_{n}^{(\alpha, \beta)}}(\la) & = & \sum_{k=1}^{n}  (-1)^{n-k}
\frac{(n + \alpha + \beta+1)_{n-k}}{(2 \imath \la)^{n-k}} a_{k}^{(\alpha +n-k, \beta + n-k)}(\la) \\
& + & (-1)^{n} \frac{(n + \alpha + \beta +1)_{n}}{(2 \imath \la)^{n}} \widehat{P_{0}^{(\alpha + n, \beta + n)}}(\la).
\end{eqnarray*}
\noindent
Evaluate the last term is evaluated as $a_{0}^{(\alpha,\beta)}(\la)$ and use 
\begin{equation}
P_{n}^{(\alpha, \beta)}(1) = \binom{\alpha+n}{n} \text{ and }
P_{n}^{(\alpha, \beta)}(-1) = (-1)^{n} \binom{\beta +n}{n} \ch
\end{equation}
\noindent
from \eqref{jacobi-111} to obtain 
\begin{equation}
a_{n}^{(\alpha, \beta)} = \frac{1}{\imath \la} \left[ e^{\imath \la} \binom{\alpha+n}{n} - 
(-1)^{n} e^{- \imath \la} \binom{\beta+n}{n} \right].
\label{boundary-1}
\end{equation}
\noindent
Some algebraic simplification now gives the stated result.   The value for $\la = 0$ comes 
directly from \eqref{der-jacobi}.
\end{proof}

The next statement represents a hypergeometric rewrite of the last 
formula in Theorem \ref{fourier-jacobi}.

\begin{Thm}
The finite Fourier transform of the Jacobi polynomial is given by
\begin{eqnarray*}
\widehat{P_{n}^{(\alpha, \beta)}}(\la) & = & 
\frac{( \beta+1)_{n}}{\imath \la n!} 
(-1)^{n+1} e^{- \imath \la}  \pFq31{n+ \alpha+ \beta + 1,  -n, 1}{\beta + 1}{\frac{-1}{2 \imath \la}} +  \ch \\
& & \qquad + \frac{( \alpha+1)_{n}}{\imath \la n!} 
e^{\imath \la}  \pFq31{n+ \alpha + \beta +1,  -n, 1}{\alpha + 1}{\frac{1}{2 \imath \la}}.
\end{eqnarray*}
\end{Thm}
\begin{proof}
The first term in the expression the last formula of Theorem \ref{fourier-jacobi} is simplified using 
\eqref{poch-113} and 
$(\beta + k + 1)_{n-k} = \frac{(\beta+1)_{n}}{(\beta+1)_{k}} $
to obtain
\begin{multline*}
 \frac{(-1)^{n-k} (n+ \alpha + \beta+1)_{k} (\beta+k+1)_{k}}
{(-2 \imath \la)^{k+1} (n-k)!}  =  \\
\frac{(-1)^{n+1} (\beta+1)_{n}}{2  \imath \la} 
\frac{(n + \alpha + \beta +1)_{k} (-n)_{k} (1)_{k}}{(\beta+1)_{k}} 
\frac{t^{k}}{k!}
\end{multline*}
\noindent
with $t = -1/2 \imath \la$. Summing from $k=0$ to $n$ gives the  first term
in the answer. A similar argument simplifies  the second term in Theorem 
\ref{fourier-jacobi}. 
\end{proof}

%
%
%

\begin{Note}
Define 
\begin{equation}
A_{n}^{(a,b)}(t) = \frac{(a+1)_{n}}{n!} 
\pFq31{n+a+b+1,-n,1}{a+1}{ \frac{1}{t}}.
\end{equation}
then the finite Fourier transform of the Jacobi 
polynomial $P_{n}^{(\alpha, \beta)}(x)$ is given by
\begin{equation}
\widehat{P_{n}^{(\alpha,\beta)}}(\la) = 
\frac{1}{\imath \la} \left[ (-1)^{n+1}e^{- \imath \la} A_{n}^{(\beta,\alpha)}(-2 \imath \la) 
+ e^{\imath \la} A_{n}^{(\alpha, \beta)}(2 \imath \la) \right].
\end{equation}
\end{Note}

\section{A collection of special examples}
\label{sec-examples}

This section presents a collection of special cases of the Jacobi polynomials 
and their respective finite Fourier transforms.

\subsection{ Legendre polynomials}

These polynomials were discussed in Section \ref{sec-jacobi} and correspond to 
the special case $\alpha = \beta =0$; that is, 
\begin{equation}
P_{n}(x) = P_{n}^{(0,0)}(x).  \ch
\end{equation}

The first formula in Theorem \ref{fourier-jacobi} reproduces the third formula in 
Theorem \ref{fourier-legendre}. Similarly, the second formula in Theorem \ref{fourier-jacobi}
gives the last expression for the finite Fourier transform of Legendre polynomials in 
Theorem \ref{fourier-legendre}. 

\subsection{Gegenbauer polynomials}
\label{sec-gegenbauer}

These polynomials are also special cases of $P_{n}^{(\alpha, \beta)}(x)$:
\begin{equation}
C_{n}^{(\nu)}(x) = \frac{(2 \nu)_{n}}{(\nu + 1/2)_{n}} P_{n}^{(\nu - 1/2, \nu- 1/2)}(x). \ch
\end{equation}

\begin{theorem}
The finite Fourier transform of the Gegenbauer polynomial $C_{n}^{(\nu)}(x)$ is given by
\begin{eqnarray*}
\widehat{C_{n}^{(\nu)}}(\la) & = & 
2  (2 \nu)_{n} e^{\imath \la} 
\sum_{k=0}^{n}  2^{2k}
\frac{(n + 2 \nu)_{k} (\nu)_{k}}{(n-k)! (2 \nu)_{2k}}
\left[\frac{ e^{-2 \imath \la} E_{k}(2 \imath \la) - 1}{(-2 \imath \la )^{k+1}}
\right]   \ch \\
& = & 
\frac{ 2 (2 \nu)_{n} (\nu)_{n}}{( 2 \nu)_{2n}} 
 \sum_{k=0}^{n} 2^{2k} \frac{(n + 2 \nu)_{k} (2 \nu + 2k)_{2n-2k}}{(n-k)! (\nu + k)_{n-k}} 
\left[ \frac{(-1)^{n-k} e^{- \imath \la} - e^{\imath \la} }{(- 2 \imath \la)^{k+1}} \right]  
\ch 
\end{eqnarray*}
\noindent
and also 
\begin{multline*}
\widehat{C_{n}^{(\nu)}}(\la) =  \frac{(2 \nu)_{n}}{\imath \la n!} \times 
\left[ (-1)^{n+1} e^{- \imath \la} \pFq31{n+ 2 \nu, -n, 1}{\nu + \tfrac{1}{2}}{- \frac{1}{2 \imath \la} } + \right. \\
 \left. e^{\imath \la} \pFq31{n+ 2 \nu, -n, 1}{\nu + \tfrac{1}{2}}{\frac{1}{2 \imath \la} } \right].
 \end{multline*}
 \end{theorem}

\subsection{Chebyshev polynomials}
\label{sec-chebyshev}

The Chebyshev polynomial are related to Gegenbauer polynomials by 
\begin{equation}
U_{n}(x)  = C_{n}^{(1)}(x)  \text{ and }
T_{n}(x) = \lim\limits_{\nu \to 0} \frac{n C_{n}^{(\nu)}(x)}{2 \nu}, \text{ for } 
n \geq 1.  \ch
\end{equation}
\noindent 
These formulas are now used to evaluate the finite Fourier transform of 
Chebyshev polynomials.

\begin{Thm}
The finite Fourier transform of the Chebyshev polynomial is given by 
\begin{eqnarray*}
\widehat{U_{n}}(\la)  & = & 
e^{\imath \la} \sum_{k=0}^{n} 2^{2k+1}k! \binom{n+k+1}{n-k} 
\left[ \frac{e^{-2 \imath \la}E_{k}(2 \imath \la) -1}{(- 2 \imath \la)^{k+1}} \right]  \ch \\
 & = & \sum_{k=0}^{n}  \frac{2^{2k+1}(n+k+1)! \, k!}{(2k+1)! \, (n-k)!} 
\frac{ \left[ (-1)^{n-k} e^{ - \imath \la} - e^{\imath \la} \right]}{( - 2 \imath \la)^{k+1}}  \ch 
\end{eqnarray*}
\noindent
and 
\begin{eqnarray*}
\widehat{T_{n}}(\la) & = & 
\sum_{k=0}^{n} (-1)^{k+1}\frac{n2^{k} (n+k)! k!}{(n-k)! (2k)! (n+k)} 
\frac{ \left[(-1)^{n-k} e^{- \imath \la} - e^{\imath \la} \right]}{( \imath \la)^{k+1}} 
\end{eqnarray*}
\end{Thm}

\section{Biorthogonality for the Jacobi polynomials}
\label{sec-biorthogonality}

The sequence of functions $\{ \frac{1}{\sqrt{2}} e^{\pi \imath j x}: \, j \in \mathbb{Z} \}$ forms an 
orthonormal  family on the Hilbert space $L^{2}[-1,1]$. Therefore, every continuous function  $f$
defined on $[-1,1]$ may be expanded in the form 
\begin{equation}
f(x) = \frac{1}{\sqrt{2}} \sum_{j= - \infty}^{\infty} a_{j}(f) e^{\pi \imath j x},
\end{equation}
\no indent
where the Fourier coefficients are given by 
\begin{equation}
a_{j}(f) = \frac{1}{\sqrt{2}} \int_{-1}^{1} f(x) e^{-\pi \imath j x} \, dx.
\end{equation}
Parseval's identity \cite[Theorem $14$]{hardy-1950a}  states that 
\begin{equation}
\int_{-1}^{1} f(x) \overline{g(x)} \, dx = \sum_{j=-\infty}^{\infty} a_{j}(f) \overline{a_{j}(g)}.
\end{equation}

This identity is now made explicit for the case 
\begin{equation}
f(x) = P_{n}^{(\alpha,\beta)}(x) \text{ and } g(x) = Q_{n}^{(\alpha,\beta)}(x) := 
(1-x)^{\alpha}(1+x)^{\beta} P_{n}^{(\alpha,\beta)}(x). 
\end{equation}

The Fourier coefficients $a_{j}(Q_{m}^{(\alpha,\beta)}(x))$  are given in 
\eqref{jacobi-101} and $a_{j}(P_{n}^{(\alpha,\beta)}(x))$ have been evaluated in Theorem \ref{fourier-jacobi}. Parseval's identity 
and the orthogonality of Jacobi polynomials give
\begin{eqnarray*}
\sum_{j = -\infty}^{\infty} a_{j}(P_{n}^{(\alpha,\beta)}(x)) \overline{  a_{j}(Q_{m}^{(\alpha,\beta)}(x)) }
& = & \frac{2^{\alpha + \beta +1} \Gamma(n + \alpha + 1) \Gamma(n + \beta+1)}{
(2n+\alpha+\beta+1) n! \Gamma(n+ \alpha + \beta + 1)}  \delta_{n,m}, \nonumber 
\end{eqnarray*}
\noindent
where $\delta_{n,m}$ is Kronecker's delta ($1$ if $n=m$ and $0$ if $n \neq m$).  Only the case 
$n \neq m$ leads to an interesting relation.  A direct calculation shows that $a_{0}(Q_{m}^{(\alpha, \beta)}(x) = 0$, so that 
Parseval's identity is written as 
\begin{eqnarray*}
\sum_{j \in \mathbb{Z} , j \neq 0}  a_{j}(P_{n}^{(\alpha,\beta)}(x)) \overline{  a_{j}(Q_{m}^{(\alpha,\beta)}(x)) } = 0, \quad 
\text{ for } n \neq m.
 \end{eqnarray*}

To simplify the previous relation, replace $\la = - \pi j$ in \eqref{jacobi-101} and use  Kummer's identity 
\begin{equation}
\pFq11{u}{u+v}{z} = e^{z} \pFq11{v}{u+v}{-z}
\end{equation}
\noindent
to obtain 
\begin{multline*}
\overline{ a_{j}(Q_{m}^{(\alpha, \beta)})} = 
\frac{(-1)^{j} j^{m}}{m!} 2^{m + \alpha + \beta + 1/2} B(m + \alpha + 1, m + \beta + 1)  \\
\pFq11{m+\beta +1}{2m + \alpha + \beta + 2}{ 2 \pi \imath j}. 
\end{multline*}
Similiarly,  \eqref{fourier-jacobi} with $\la = - \pi j$ gives 
\begin{multline*}
a_{j}(P_{n}^{(\alpha, \beta)}) = \frac{(-1)^{j}}{2 \pi \imath j n!}  \\
\left[ (-1)^{n} (\beta+1)_{n}  \,\, \pFq31{ n + \alpha + \beta + 1, -n, 1}{\beta +1 }{ \frac{1}{2 \pi \imath j}} - \right.  \\
\left.  (\alpha+1)_{n} \,\, \pFq31{ n + \alpha + \beta + 1, -n, 1}{\alpha +1 }{ -\frac{1}{2 \pi \imath j}}  \right].
 \end{multline*}
 Parseval's identity now produces the next result. 
 
 \begin{theorem}
 Define 
 \begin{equation*}
 W_{n,m}^{(\alpha,\beta)}(t;j) = (\alpha +1)_{n} j^{m-1}
 \pFq31{n+ \alpha + \beta + 1,-n,1}{\alpha + 1}{ \frac{1}{t}} 
 \pFq11{m + \alpha + 1}{2m + \alpha + \beta + 2}{t}.
 \end{equation*}
 \noindent
 Then 
 \begin{equation}
 (-1)^{n} \sum_{j \in \mathbb{Z}, \, j \neq 0} W_{n,m}^{(\beta, \alpha)}(2 \pi \imath j;j)= 
(-1)^{m-1}  \sum_{j \in \mathbb{Z}, \, j \neq 0} W_{n,m}^{(\alpha,\beta)}(2 \pi \imath j;j).
 \end{equation}
 In particular, if $n$ and $m$ have opposite parity, then 
 \begin{equation}
  \sum_{j \in \mathbb{Z}, \, j \neq 0} W_{n,m}^{(\beta, \alpha)}(2 \pi \imath j;j) = 
    \sum_{j \in \mathbb{Z}, \, j \neq 0} W_{n,m}^{(\alpha,\beta)}(2 \pi \imath j;j).
 \end{equation}
 \end{theorem}

\section{An operator point of view}
\label{sec-operator}

To obtain the  finite Fourier transform of a polynomial start with 
\begin{equation}
\int_{-1}^{1} x^{k} e^{\imath \la x} \, dx = (-\imath D)^{k} (2 \text{ sinc }\la)
\end{equation}
\noindent
where the \textit{sinc} function is 
\begin{equation}
\text{sinc }\la = \frac{\sin \la}{\la}
\end{equation}
and $D = \frac{d}{d \la}$. The action is extended by linearity to obtain
\begin{equation}
\widehat{P}(\la) = P(- \imath D) (2 \text{ sinc }\la).
\end{equation}

For instance, for the Chebyshev polynomial 
\begin{equation}
U_{n}(x) = \sum_{k=0}^{n} (-2)^{k} \binom{n+k+1}{n-k} (1-x)^{k} 
\end{equation}
\noindent
leads to 
\begin{eqnarray}
\widehat{U_{n}}(\la) & = & \sum_{k=0}^{n} (-2)^{k} \binom{n+k+1}{n-k} 
(1 + \imath D)^{k} (2 \text{ sinc  }\la)  \\
& = & U_{n}(- \imath D) ( 2 \text{ sinc }\la). \nonumber 
\end{eqnarray}

It is elementary to check that 
\begin{equation}
\left( \frac{d}{d \la} \right)^{n}  \text{ sinc }\la = 
A_{n}(\la) \sin \la + B_{n}(\la) \cos \la
\end{equation}
\noindent
where $A_{n}, \, B_{n}$ are polynomials in $1/\la$ that satisfy the recurrences
\begin{eqnarray*}
A_{n+1}(\la) & = & A_{n}'(\la) - B_{n}(\la) \\
B_{n+1}(\la) & = & A_{n}(\la) + B_{n}'(\la),
\end{eqnarray*}
\noindent
with initial values $A_{0}(\la) = 1/\la$ and $B_{0}(\la) = 0$.  An explicit expression for these 
polynomials can be obtain from 
\begin{equation}
\left( \frac{d}{d \la} \right)^{n}  \text{ sinc }\la = 
\sum_{j=0}^{n} \frac{n!}{(n-j)!} \frac{\sin( \la + (n+j) \tfrac{\pi}{2})}{\la^{j+1}}.
\end{equation}
 Details of this 
approach to finite Fourier transform of orthogonal polynomials will be 
given elsewhere.

\medskip

\noindent
{\bf Acknowledgments}. 
The authors wish to thank A.~Fokas, T. Koorwinder and T. Amdeberhan for 
discussions on the topic presented here. The work of the fourth author was partially funded by
$\text{NSF-DMS } 1112656$.  The first author is a postdoctoral fellow and the second 
author is a graduate student partially funded by the same grant.  \\

\end{document}